\documentclass[11pt]{article}

\usepackage{amsmath}
\usepackage{amssymb}
\usepackage{amsthm}
\usepackage{multirow}
\usepackage{color}

\newtheorem{theorem}{Theorem}[section]
\newtheorem{remark}[theorem]{Remark}
\newtheorem{definition}{Definition}[section]
\newtheorem{lemma}[theorem]{Lemma}
\newtheorem{example}[theorem]{Example}

\newtheorem{corollary}[theorem]{Corollary}
\newtheorem{proposition}[theorem]{Proposition}

\newcommand{\ZZ}{\mathbb{Z}}
\newcommand{\CC}{\mathbb{C}}
\newcommand{\F}{\mathbb{F}}

\newcommand{\GR}{GR(4,n)}
\newcommand{\Om}{\Omega}
\newcommand{\x}{x+2\sqrt{f(x)}}
\newcommand{\y}{y+2\sqrt{f(y)}}

\newcommand{\X}{\chi}
\newcommand{\E}{\mathcal{E}}
\newcommand{\R}{\mathcal{R}}
\newcommand{\Sc}{\mathcal{S}}
\newcommand{\Tr}{\textup{Tr}}
\newcommand{\N}{\textup{N}}
\newcommand{\ord}{\text{ord}\,}
\newcommand{\Aut}{{\text{Aut}}}

\def\abs#1{{\vert{#1}\vert}}
\newcommand{\Iff}{if and only if }

\title{New Pseudo-Planar Binomials in Characteristic Two\\ and Related Schemes}

\author{
Sihuang Hu$^{\text{a}}$, Shuxing Li$^{\text{a}}$, Tao Zhang$^{\text{a}}$,
Tao Feng$^{\text{a}}$ and Gennian Ge$^{\text{b,c,}}$\thanks{Corresponding author. Email address: gnge@zju.edu.cn}\\
  \footnotesize $^{\text{a}}$ Department of Mathematics, Zhejiang University, Hangzhou 310027, Zhejiang, China\\
  \footnotesize $^{\text{b}}$ School of Mathematical Sciences, Capital Normal University, Beijing, 100048, China\\
\footnotesize $^{\text{c}}$ Beijing Center for Mathematics and Information
Interdisciplinary Sciences,
Beijing, 100048, China.}

\begin{document}
\date{}\maketitle

\begin{abstract}
  Planar functions in odd characteristic were introduced by Dembowski and Ostrom
  in order to construct finite projective planes in 1968. They were also used in
  the constructions of DES-like iterated ciphers, error-correcting codes, and signal
  sets.
  Recently, a new notion of pseudo-planar functions in even characteristic was proposed by Zhou.
  These new pseudo-planar functions, as an analogue of planar functions in odd characteristic,
  also bring about finite projective planes. There are three known infinite families of pseudo-planar monomial functions constructed by Schmidt and Zhou, and Scherr and Zieve.
  In this paper, three new classes of pseudo-planar binomials are provided.
  Moreover, we find that each pseudo-planar function gives an association scheme which is defined on a Galois ring.

\smallskip

\end{abstract}


\section{Introduction}\label{sec:intro}

Let $q=p^n$ where $p$ is an odd prime and $n$ is a positive integer. A function $f:\F_{q}\rightarrow\F_{q}$ is {\em planar} if the mapping
\begin{equation}
\label{def}
x\rightarrow f(x+\epsilon)-f(x)
\end{equation}
is a permutation of $\F_{q}$ for each $\epsilon\in\F_{q}^*$.
Planar functions were introduced by Dembowski and Ostrom~\cite{DO} to construct
finite projective planes over finite fields with odd characteristic.
Apart from this, planar functions emerge from many other applications.
In the cryptography literature, they are called {\em perfect nonlinear functions}~\cite{NK},
and used in the constructions of DES-like iterated ciphers, since
they are optimally resistant to differential cryptanalysis.
Carlet, Ding, and Yuan~\cite{CDY,D,YCD}, among others, utilized planar functions
to construct error-correcting codes, which are then employed to design secret sharing schemes.
Planar functions are also applied to the construction of authentication codes~\cite{DN}, constant composition codes~\cite{DY1} and signal sets~\cite{DY}. Besides, planar functions induce many combinatorial objects such as skew Hadamard difference sets and Paley type partial difference sets~\cite{WQWX}.

When $p=2$, there are no planar functions over $\F_{2^n}$, since if $x$ satisfies
$f(x+\epsilon)-f(x)=d$, then so does $x+\epsilon$. As an alternative, a function $f:\F_{2^n}\rightarrow\F_{2^n} $ is said to be {\em almost perfect nonlinear} if the mapping (\ref{def}) is $2$-to-$1$ for every $\epsilon\in\F_{2^n}^*$. However, there is no apparent link between almost perfect nonlinear functions and finite projective planes. Recently, Zhou~\cite{Z} put forward
a definition of ``planar" functions over finite fields with characteristic two, which give rise to
finite projective planes. From now on, we call a function $f:\F_{2^n}\rightarrow\F_{2^n}$
{\em pseudo-planar} if
$$
x\rightarrow f(x+\epsilon)+f(x)+\epsilon x$$
is a permutation on $\F_{2^n}$ for each $\epsilon\in\F_{2^n}^*$.
Note that Zhou~\cite{Z} called such functions ``planar",
and the term ``pseudo-planar" was first used by Abdukhalikov~\cite{Abdukhalikov}
to avoid confusion with planar functions in odd characteristic.

The pseudo-planar monomial functions have been investigated by Schmidt and Zhou~\cite{SZh2013},
and Scherr and Zieve~\cite{SZi2013}.
They are listed in Table \ref{tab:knownPlanarMonomials},
where $\Tr_{n/2}$ denotes the trace function from $\F_{2^{n/2}}$ to $\F_2$.
\begin{table}
\begin{center}
\caption{The known pseudo-planar monomials on $\F_{2^n}$}
\label{tab:knownPlanarMonomials}
\begin{tabular}{ccc}
\hline
Function  &  Condition & Reference \\ \hline
$a x^{2^k}$  &  $a \in \F_{2^n}^*$  & trivial \\
$a x^{2^k+1}$     &  $n=2k, a \in \F_{2^{n/2}}^*, \Tr_{n/2}(a)=0$ & \cite[Theorem 6]{SZh2013}\\
\multirow{2}{*}{$a x^{4^k(4^k+1)}$} & $n=6k$, $a \in \F_{2^n}^*$, $a$ is a $(4^k-1)$-th  & \multirow{2}{*}{\cite[Theorem 1.1]{SZi2013}} \\
                                    & power but not a $3(4^k-1)$-th power & \\ \hline
\end{tabular}
\end{center}
\end{table}
In this paper, we construct three new classes of pseudo-planar binomial functions,
at least two of them are infinite families.
Association schemes form a central part of algebraic combinatorics, and
play important roles in several branches of mathematics, such as coding
theory and graph theory.
One interesting result we obtained is that pseudo-planar
functions will always give 5-class association schemes which are
defined on Galois rings. Our construction can be regarded as an
analogue of the one studied by Liebler and Mena~\cite{LM}, and
Bonnecaze and Duursma~\cite{BD}. Similar (but symmetric) 4-class
association schemes were constructed by Abdukhalikov, Bannai and
Suda~\cite{ABS}, and LeCompte, Martin and Owens~\cite{LMO}.
Analogous to the case of almost perfect nonlinear functions, we
define the Fourier spectrum of pseudo-planar functions. With the
information obtained from eigenmatrices of those association
schemes, we completely determine the Fourier spectrum.

The rest of this paper is organized as follows. Section~\ref{sec:preliminary} contains the background of the  mathematical objectives involved. Section~\ref{sec:constructions} presents the construction of three classes of pseudo-planar binomial functions. Section~\ref{sec:assosche} investigates the association schemes arising from pseudo-planar functions. Section~\ref{sec:concluding} concludes this paper.

\section{Preliminaries}\label{sec:preliminary}

\subsection{Relative difference sets and the inversion formula}

Let $G$ be a finite abelian group and let $N$ be a subgroup of $G$.
A subset $D$ of $G$ is a \emph{relative difference set (RDS)} with
parameters $(\abs{G}/\abs{N},\abs{N},\abs{D},\lambda)$ and \emph{forbidden subgroup} $N$ if the list of nonzero differences of $D$ comprises every element in $G\setminus N$ exactly $\lambda$ times, and no element of $N\backslash\{0\}$. The \emph{group ring} $\ZZ[G]$ is a free abelian group with a basis $\{g \mid g \in G\}$. For any set $A$ whose elements belong to $G$ ($A$ may be a multiset), we identify $A$ and the group ring element $\sum_{g \in A} d_g g$ throughout the rest of the paper, where $d_g$ is the multiplicity of $g$ appears in $A$. Given any $A=\sum d_g g \in \ZZ[G]$, we define $A^{(-1)}=\sum d_g g^{-1}$, in which $g^{-1}$ is the inverse of $g$ with respect to the operation of group $G$. Using the language of group ring, a relative difference set $D$ in $G$ with forbidden group $N$ can be expressed in a succinct way:
$$
DD^{(-1)}=|D| 1_G + \lambda(G - N),
$$
where $1_G$ is the identity of group $G$.

For a finite abelian group $G$, denote its character group by $\widehat{G}$.
For any $A=\sum d_g g$ and $\X \in \widehat{G}$, define $\X(A)=\sum d_g \X(g)$.
The following \emph{inversion formula} shows that $A$ is completely determined
by its character value $\X(A)$, where $\X$ ranges over $\widehat{G}$.
For convenience, we will denote $d_{1_G}$ by $[A]_0$ throughout this paper.

\begin{lemma}\label{inversion formula}
Let $G$ be an abelian group. If $A=\sum_{g\in G}d_g
g\in \ZZ[G]$, then
$$
d_h=\frac{1}{|G|}\sum_{\chi\in\widehat G}\chi(A) \chi (h^{-1}),
$$
for all $h\in G$. In particular, we have
$$
[A]_0=\frac{1}{|G|}\sum_{\chi\in\widehat G}\chi(A).
$$
\end{lemma}

\subsection{Galois rings}

We give a brief introduction to the \emph{Galois ring} $\GR$. Let $R=\GR$, then the additive group of $R$ can be identified with the abelian group $(\ZZ_4^n,+)$. Let $Z=\{2x \mid x \in R\}$, then $Z$ consists of $0$ and the zero divisors of $R$, where $0$ is the identity with respect to the addition. The unit group $R \setminus Z$ contains a cyclic subgroup of order $2^n-1$ generated by an element $\xi$. The set $T=\{\xi^i \mid 0 \le i \le 2^n-2\} \cup \{0\}$ is called \emph{Teichm\"{u}ller system}. For any $x \in R$, there exists a unique representation
\begin{equation}
\label{eqn}
x=a+2b,
\end{equation}
where $a,b \in T$. For any $x \in R$, write $\sqrt{x}$ for $x^{2^{n-1}}$. If we define the addition on $T$ by
$$
x \oplus y = x+y+2\sqrt{xy},
$$
then $(T,\oplus,\cdot)$ is a finite field with $2^n$ elements. Hence, a pseudo-planar function over $\F_{2^n}$ can also be identified with a function from $T$ into itself.
For any $x \in R$, we have $x=a+2b$ for some $a,b \in T$. The map
$$
\sigma:a+2b\mapsto a^2+2b^2
$$
is the Frobenius map of $R$, which is a ring automorphism. For any $a \in R$, the trace function of $R$ is the map $\Tr:R\rightarrow \ZZ_4$ defined by
$$
\Tr(a)=\sum_{i=0}^{n-1} \sigma^i(a).
$$
Let $\mathbf{i}=\sqrt{-1}$. For any $a \in R$, define the map $\X_a:R\rightarrow\CC$ by
$$
\X_a(x)=\mathbf{i}^{\Tr(ax)}, \quad \forall\, x \in R.
$$
Then the character group $\widehat{R}=\{\X_a \mid a \in R \}$.
For more information on Galois rings, please refer to \cite{HKCSS,LM,Wan}.

\subsection{Association schemes}

Let $X$ be a nonempty finite set. Let $R_0,R_1,\cdots, R_d$ be a partition of $X\times X$ satisfying that
\begin{enumerate}
\item[(i)] $R_0=\{(x,x)\mid x\in X\}$;
\item[(ii)] for any $0 \le i \le d$, there exists $0 \le i^{\prime} \le d$ such that $R_{i^{\prime}}=\{(y,x) \mid (x,y ) \in R_i\}$.
\end{enumerate}
For each $R_i$, its adjacency matrix is denoted by $A_i$, whose $(x,y)$-th entry is $1$ if $(x,y)\in R_i$ and $0$ otherwise. We call $(X,\{R_i\}_{i=0}^d)$ a {\em $d$-class association scheme} if there exist nonnegative integers $p_{i,j}^k$ such that
\[
A_iA_j=\sum_{k=0}^dp_{i,j}^kA_k,
\]
where $0 \le i,j,k \le d$.
The $\CC$-linear span of $A_0,A_1,\cdots,A_d$ forms a semisimple algebra of dimension $d+1$. Hence, there exists another basis $\{E_0,E_1,\cdots,E_d\}$ consisting of pairwise orthogonal idempotents.
So we have
$$A_i = \sum_{j=0}^d P_{ji}E_j$$
and
$$E_i=\frac{1}{|X|}\sum_{j=0}^{d}Q_{ji}A_j$$
for certain complex numbers $P_{ji}, Q_{ji}$. The matrix
$P=(P_{ji})$ (resp. $Q=(Q_{ji})$) is called the first (resp. second) eigenmatrix.
Clearly, we have $PQ=|X|I$, where $I$ denotes the identity matrix of order $|X|$.

Let $\{S_i \mid 0 \le i \le d\}$ be a partition of $X$. It induces a partition $\{R_i \mid 0 \le i \le d\}$ on $X \times X$ with
$$
R_i=\{(x,y) \mid x-y \in S_i\}.
$$
If $(X,\{R_i\}_{i=0}^d)$ forms an association scheme, then we call $(X,\{S_i\}_{i=0}^d)$ a \emph{Schur ring}.

Assume that $(X,\{S_i\}_{i=0}^d)$ is a Schur ring. There is an equivalence relation defined on the character group $\widehat{X}$ of $X$ as follows: $\chi\sim\chi'$ if and only if $\chi(S_i)=\chi'(S_i)$ for each $0\leq i\leq d$. Denote by $T_0, T_1,\cdots,T_d$ the equivalence classes, with $T_0$ consisting of only the principal character. Then $(\widehat{X},\{T_i\}_{i=0}^d)$ also forms a Schur ring, called the {\em dual} of $(X,\{S_i\}_{i=0}^d)$. The first eigenmatrix of the dual scheme is equal to the second eigenmatrix of the original scheme. Please refer to \cite{BI} or \cite{BCN} for more details.


We shall need the following well-known criterion due to Bannai \cite{Ban}
and Muzychuk \cite{Mu}.
\begin{theorem}[Bannai-Muzychuk criterion]
  Let $P$ be the first eigenmatrix of an association scheme $(X, \{R_i\}_{0\leq i\leq d})$, and $\Lambda_0:=\{0\}, \Lambda_1,\ldots ,\Lambda_{d'}$ be a partition of $\{0,1,\ldots ,d\}$. Then $(X, \{R_{\Lambda_i}\}_{0\leq i\leq d'})$ forms an association scheme if and only if there exists a partition $\{\Delta_i\}_{0\leq i\leq d'}$ of $\{0,1,2,\ldots ,d\}$ with $\Delta_0=\{0\}$ such that each $(\Delta_i, \Lambda_j)$-block of $P$ has a constant row sum. Moreover, the constant row sum of the $(\Delta_i, \Lambda_j)$-block is the $(i,j)$-th entry of the first eigenmatrix of the fusion scheme.
\end{theorem}

\section{Pseudo-planar binomials}\label{sec:constructions}

It is well-known that every function from $\F_{2^n}$ to itself can be
uniquely written as a polynomial function of degree at most $2^n-1$.
The monomial functions $x\mapsto cx^t$ for some $c\in\F_{2^n}$ and
some integer $t$ are the simplest nontrivial polynomial functions.
An integer $t$ satisfying that $1\le t\le 2^n-1$ is a {\em pseudo-planar exponent}
of $\F_{2^n}$ if the function $x\mapsto cx^t$ is pseudo-planar on $\F_{2^n}$
for some $c\in\F_{2^n}^*.$ The pseudo-planar monomials were first
investigated by Schmidt and Zhou~\cite{SZh2013}, and subsequently by Scherr
and Zieve~\cite{SZi2013}. Moreover, in \cite[Conjecture 8]{SZh2013},
it is conjectured that the only exponents that give pseudo-planar monomials
are those listed in Table~\ref{tab:knownPlanarMonomials}.

Besides pseudo-planar monomial functions, the next simplest cases are pseudo-planar
binomials. In this section, we construct three classes of pseudo-planar
binomials on the field $\F_{2^{3m}}$. The following result will be
useful.

\begin{lemma}[{\cite[p. 362]{LN}}]
  Let $q$ be a prime power and $\F_{q^r}$ be an extension of $\F_q$. Then the linearized
  polynomial
  $$ L(x)=\sum_{i=0}^{r-1} c_{i}x^{q^i} \in \F_{q^r}[x]$$
  is a permutation of $\F_{q^r}$ if and only if
  \begin{equation*}
    \det\left(
    \begin{array}{ccccc}
      c_0 & c_{r-1}^q & c_{r-2}^{q^2} &\cdots &c_1^{q^{r-1}}\\
      c_1 & c_{0}^q & c_{r-1}^{q^2} &\cdots &c_2^{q^{r-1}}\\
      c_2 & c_{1}^q & c_{0}^{q^2} &\cdots &c_3^{q^{r-1}}\\
      \vdots & \vdots & \vdots &   & \vdots \\
      c_{r-1} & c_{r-2}^q & c_{r-3}^{q^2} &\cdots &c_0^{q^{r-1}}\\
    \end{array}\right)\ne 0.
  \end{equation*}
  \label{lem:criterion}
\end{lemma}

Let $m$ be a positive integer. The relative trace (resp. norm) from
$\F_{2^{3m}}$ to $\F_{2^m}$ is denoted by $\Tr_3$ (resp. $\N_3$) from now on.

\begin{proposition}
\label{prop:construction2}
Suppose $m$ is an even positive integer, then the function
$$
f(x)=a^{{2^{2m}}+1}x^{{2^{2m}}+1}+a^{-({2^m}+1)}x^{{2^m}+1}
$$
is pseudo-planar on $\F_{2^{3m}}$ if and only if
$$\Tr_3((a^{{2^{2m}}+{2^m}}+a^{-{2^{2m}}-{2^m}-2})(a^{{2^m}+1}+\epsilon^{{2^m}-1})\epsilon^{{2^m}+2}+a^{{2^m}-{2^{2m}}}\epsilon^3+\epsilon) \ne 0$$
 for all $\epsilon \in \F_{2^{3m}}^*$.
\end{proposition}

\begin{proof}
Set $t=2^m$. For each $\epsilon \in \F_{2^{3m}}^{*}$,
\begin{align*}
f(x+\epsilon)+f(x)+\epsilon x
=&a^{t^2+1}\epsilon x^{t^2}+a^{-(t+1)}\epsilon x^t+(a^{t^2+1}\epsilon^{t^2}+a^{-(t+1)}\epsilon^t+\epsilon)x+
(a\epsilon)^{t^2+1}+(a^{-1}\epsilon)^{t+1}.
\end{align*}
Then it suffices to show that the polynomial
$$
G_\epsilon(x):=a^{t^2+1}\epsilon x^{t^2}+a^{-(t+1)}\epsilon x^t+(a^{t^2+1}\epsilon^{t^2}+a^{-(t+1)}\epsilon^t+\epsilon)x
$$
is a permutation on $\F_{2^{3m}}$ for any $\epsilon \in \F_{2^{3m}}^*$.
By Lemma~\ref{lem:criterion}, we see that $G_\epsilon(x)$ is a permutation
if and only if
\begin{align*}
&  \det\left(
  \begin{array}{ccc}
    a^{t^2+1}\epsilon^{t^2}+a^{-(t+1)}\epsilon^t+\epsilon &a^{t+1}\epsilon^t & a^{-(t^2+1)}\epsilon^{t^2}\\
    a^{-(t+1)}\epsilon &a^{t+1}\epsilon+a^{-(t^2+t)}\epsilon^{t^2}+\epsilon^t& a^{t^2+t}\epsilon^{t^2} \\
    a^{t^2+1}\epsilon&a^{-(t^2+t)}\epsilon^t &a^{t^2+t}\epsilon^t+a^{-(t^2+1)}\epsilon+\epsilon^{t^2} \\
  \end{array}\right)\\
=&\Tr_3((a^{t^2+t}+a^{-t^2-t-2})(a^{t+1}+\epsilon^{t-1})\epsilon^{t+2}+a^{t-t^2}\epsilon^3+\epsilon)\\
=&\Tr_3((a^{{2^{2m}}+{2^m}}+a^{-{2^{2m}}-{2^m}-2})(a^{{2^m}+1}+\epsilon^{{2^m}-1})\epsilon^{{2^m}+2}+a^{{2^m}-{2^{2m}}}\epsilon^3+\epsilon)\\
\ne& 0.
\end{align*}
This finishes the proof.
\end{proof}

\begin{remark}
We are unable to simplify the necessary
and sufficient conditions in Proposition~\ref{prop:construction2} to
provide a more concise criterion. We also cannot decide whether this construction will
give infinite families of pseudo-planar binomials or not.
\end{remark}

Here we give two examples.
For any $a \in \F_{2^n}^*$, denote the multiplicative order of $a$ by $\ord(a)$.

\begin{example}
When $m=2$, direct computation via computer program shows that
$$
f(x)=a^{17}x^{17}+a^{-5}x^5
$$
is pseudo-planar on $\F_{2^{3m}}$ if and only if $\ord(a) \in \{9,63\}$, which coincides with the condition in Proposition~\ref{prop:construction2}.
\end{example}

\begin{example}
When $m=4$, direct computation via computer program shows that
$$
f(x)=a^{257}x^{257}+a^{-17}x^{17}
$$
is pseudo-planar on $\F_{2^{3m}}$ if and only if $\ord(a) \in \{9,63,117,819\}$, which coincides with the condition in Proposition~\ref{prop:construction2}.
\end{example}

In the following of this section, we give two infinite families of pseudo-planar binomials.

Let $m$ be a positive integer. Suppose $\epsilon\in \F_{2^{3m}}^*\backslash\F_{2^m}$ and its
minimal polynomial over $\F_{2^m}$ is
$$ C_{\epsilon}(x)=x^3+B_1x^2+B_2x+B_3\in\F_{2^m}[x]\quad (B_{3}\ne 0).$$
Denote the three roots of $C_{\epsilon}(x)$ by $x_1(=\epsilon)$, $x_2(=\epsilon^{2^m}),$
and $x_3(=\epsilon^{2^{2m}})$. It follows that
\begin{align*}
  B_1=& x_1+x_2+x_3=\Tr_3(\epsilon),\\
  B_2=& x_1x_2+x_1x_3+x_2x_3,\\
  B_3=& x_1x_2x_3=\N_3(\epsilon).
\end{align*}
We can verify that
\begin{align*}
  \Tr_3(\epsilon^3)=&x_1^3+x_2^3+x_3^3\\
                   =&(x_1+x_2+x_3)^3+x_1x_2x_3+(x_1+x_2+x_3)( x_1x_2+x_1x_3+x_2x_3)\\
                   =&B_1^3+B_3+B_1B_2,\\
\Tr_3(\epsilon^{1+2^{m+1}})=&\Tr_3(x_1x_2^2) =x_1x_2^2+x_2x_3^2+x_3x_1^2.
\end{align*}
Set $u_1=\Tr_3(x_1x_3^2)$ and $u_2=\Tr_3(x_1x_2^2)$.
Then we have
\begin{align}
  u_1+u_2=&
  B_3+B_1B_2,\label{eqn:u_sum}\\
  u_1u_2=&B_1^3B_3+B_2^3+B_3^2.\label{eqn:u_product}
\end{align}

We would like to point out that part of the following proof for Proposition~\ref{prop:construction3}
with $m\equiv 1\pmod 3$
 is provided by
one of the anonymous referee and communicated with the Associate Editor.

\begin{proposition}
\label{prop:construction3}
  Let $m$ be a positive integer and $m\not\equiv 2\pmod 3$. Then
  $$ f(x)=x^{2^m+1}+x^{2^{2m}+2^m} $$
  is pseudo-planar on $\F_{2^{3m}}$.
\end{proposition}

\begin{proof}
  A similar analysis as the proof of Proposition~\ref{prop:construction2}
  shows that $f$ is pseudo-planar \Iff
  $$\N_3(\epsilon)+\Tr_{3}(\epsilon^3+\epsilon^{1+2^{m+1}})\neq 0$$
  for every $\epsilon\in\F_{2^{3m}}^*$. For convenience, we write
  $M_{\epsilon}=\N_3(\epsilon)+\Tr_{3}(\epsilon^3+\epsilon^{1+2^{m+1}})$.

  First suppose $\epsilon\in\F_{2^m}^*$. Then
  $ M_{\epsilon}=\N_3(\epsilon)+\Tr_{3}(\epsilon^3+\epsilon^{3})=\N_3(\epsilon)\ne 0.$

  Now let $\epsilon\in \F_{2^{3m}}^*\backslash\F_{2^m}$.
It can be verified that
$$ M_{\epsilon}=B_1^3+B_1B_2+u_2.$$
We will split our consideration into two parts according to whether $B_1=0$ or not.

Suppose $B_1=0$. Then $M_{\epsilon}=u_2$.
Now if $M_{\epsilon}=0$, from $\eqref{eqn:u_product}$, we get $B_3=B_2^{3/2}$.
Therefore $B_2\neq 0,$ since otherwise $B_1=B_2=B_3=0$, which is impossible.
Replace $B_3=B_2^{3/2}$ into $C_{\epsilon}(x)$, we obtain
$$
\left(\frac{\epsilon}{B_2^{1/2}}\right)^3+\frac{\epsilon}{B_2^{1/2}}+1=0,
$$
which implies that
$$ \frac{\epsilon}{B_2^{1/2}}\in\F_{2^3}. $$
That is to say that $\epsilon=b\beta$ with
$\beta:=B_2^{1/2}\in\F_{2^m}^*$ and
$b:=\epsilon/{B_2^{1/2}}\in\F_{2^3}^*$. If $m\equiv 0\pmod 3$, then
$b\in\F_{2^3}^*\subseteq\F_{2^m}$, so $\epsilon\in\F_{2^m}$, which
is a contradiction. If $m\equiv 1\pmod 3$, we see that $2^m\equiv
2\pmod 7$ and $2^{m+1}\equiv 2^{2m} \equiv 4\pmod 7$. Then
\begin{align*}
    \Tr_3(\epsilon^3)=&\Tr_3( (b\beta)^3)=\beta^3\Tr_3(b^3),\\
    \Tr_3(\epsilon^{1+2^{m+1}})=&\Tr_3(b^{1+2^{m+1}}\beta^{1+2^{m+1}})=\beta^3\Tr_3(b^5)=\beta^3\Tr_3(b^3).
\end{align*}
Hence
\begin{align*}
  M_{\epsilon} =& \N_3(\epsilon)+\Tr_3((b\beta)^3+(b\beta)^{1+2^{m+1}})=\N_3(\epsilon)\ne 0
\end{align*}
which is a contradiction.

Next suppose $B_1 \neq 0$. Without loss of generality we let $B_1=1$.
Assume that $M_{\epsilon}=1+B_2+u_2=0$, then
$u_2=B_2+1$. Replace it in \eqref{eqn:u_sum} and $\eqref{eqn:u_product}$, we get $u_1=B_3+1$,
and
\begin{equation}\label{eqn:Bs}
B_2^3+B_3^2+B_2B_3+B_2+1=0.
\end{equation}
If $B_2=0$, then $B_3=1$, and
$$\epsilon^3+\epsilon^2+1=0.$$
Similarly as above, this finally leads to $M_{\epsilon}=\N_3(\epsilon)\ne0$, which contradicts the assumption that $M_{\epsilon}=0$.
If $B_2\neq0$, we write $w=(B_3+1)/B_2$. Then \eqref{eqn:Bs}
becomes $B_2=w^2+w$. Hence $B_3=B_2w+1=w^3+w^2+1$. We rewrite
$C_{\epsilon}(x)$ as
\begin{equation}
x^3+x^2+(w^2+w)x+(w^3+w^2+1)=0.
  \label{eqn:C}
\end{equation}
Let the three roots of the polynomial $x^3+x+1$ in $\F_{2^m}$ be $\tau_1$,
$\tau_2(=\tau_1^2)$, and $\tau_3(=\tau_1^{4})$. We compute that
\begin{align*}
&(\tau_2 +\tau_1w+1)^3+(\tau_2 +\tau_1w+1)^2+B_2(\tau_2 +\tau_1w+1)+B_3\\
=&(\tau_1^3+\tau_1+1)w^3+(\tau_2 \tau_1^2+\tau_2 +\tau_1)w^2+(\tau_2 ^2\tau_1+\tau_2 +\tau_1+1)w+\tau_2 ^3+\tau_2 +1\\
=&0.
\end{align*}
Therefore the element $\tau_2+\tau_1w+1$ is a root of $C_{\epsilon}(x)$.
If $m\equiv 0\pmod 3$, then $\tau_i\,(1\le i\le 3)\in\F_{2^3}\subseteq\F_{2^m}$ and hence
$\tau_2+\tau_1w+1\in\F_{2^m}$. This
contradicts the fact that $C_{\epsilon}(x)$ is irreducible over $\F_{2^m}$.
If $m\equiv 1\pmod 3$, we see that
\begin{align*}
  \Tr_3(\epsilon^3)
  =& \Tr_3((\tau_2 +\tau_1w+1)^3)\\
  =& (\tau_2 +\tau_1w+1)^3 + (\tau_2 +\tau_1w+1)^{3\cdot 2^m} + (\tau_2 +\tau_1w+1)^{3\cdot2^{2m}}\\
  =& (\tau_2 +\tau_1w+1)^3 +(\tau_3 +\tau_2 w+1)^3 + (\tau_1 +\tau_3 w+1)^3\\
  =& (\tau_1^3+\tau_2^3+\tau_3^3)w^3+(\tau_1^2\tau_2+\tau_2^2\tau_3+\tau_3^2\tau_1+\tau_1^2+\tau_2^2+\tau_3^2)w^2\\
  & + (\tau_1\tau_2^2+\tau_2\tau_3^2+\tau_3^2\tau_1^2+\tau_1+\tau_2+\tau_3)w
    +(\tau_1^3+\tau_2^3+\tau_3^3+\tau_1^2+\tau_2^2+\tau_3^2+\tau_1+\tau_2+\tau_3+1)\\
  =& w^3+w^2,\\
  \Tr_3(\epsilon^{1+2^{m+1}})
  =& \Tr_3((\tau_2 +\tau_1w+1)^{1+2^{m+1}})\\
  =& \Tr_3((\tau_2 +\tau_1w+1)(\tau_1 +\tau_3w^2+1))\\
  =& (\tau_1\tau_2+\tau_2\tau_3+\tau_3\tau_1)w^3 + (\tau_1\tau_2+\tau_2\tau_3+\tau_3\tau_1+\tau_1+\tau_2+\tau_3)w^2\\
   & + (\tau_1^2+\tau_2^2+\tau_3^2+\tau_1+\tau_2+\tau_3)w+(\tau_1\tau_2+\tau_2\tau_3+\tau_3\tau_1+1)\\
  =& w^3+w^2.
\end{align*}
Thus
\begin{align*}
  M_{\epsilon} = \N_3(\epsilon)+\Tr_{3}(\epsilon^3+\epsilon^{1+2^{m+1}})=\N_3(\epsilon)\ne 0,
\end{align*}
which is also a contradiction.
\end{proof}

\begin{remark}
    Let $m\equiv 2\pmod 3$. Suppose $\epsilon\in\F_{2^{3m}}$
    satisfying $\epsilon^3+\epsilon^2+1=0$. (It is not hard to show that such $\epsilon$ exists.)
    Then we can compute
  $M_{\epsilon}=N_3(\epsilon)+\Tr_3(\epsilon^3+\epsilon^{1+2^{m+1}})=\sum_{i=0}^{6}\epsilon^i=0$. Thus $f(x)=x^{2^m+1}+x^{2^{2m}+2^m}$
    is not pseudo-planar on $\F_{2^{3m}}$.
\end{remark}


\begin{proposition}
  Let $m$ be a positive integer and $m\not\equiv 1\pmod 3$. Then
  $$ f(x)=x^{2^{2m}+1}+x^{2^{2m}+2^m} $$
  is pseudo-planar on $\F_{2^{3m}}$.
  \label{prop:construction4}
\end{proposition}
\begin{proof}
  A similar analysis to the proof of Proposition~\ref{prop:construction2}
  shows that $f$ is pseudo-planar \Iff
  $$\N_3(\epsilon)+\Tr_{3}(\epsilon^3+\epsilon^{2+2^{m}})\neq 0$$
  for every $\epsilon\in\F_{2^{3m}}^*$. The remaining discussion is analogous to
  Proposition~\ref{prop:construction3}.
\end{proof}

\section{Association schemes arising from pseudo-planar functions}\label{sec:assosche}
Let $R=\GR$ be a Galois ring. For any set $A$ whose elements belong to $R$ ($A$ may be a multiset), we identify $A$ and the group ring element $\sum_{g \in A} d_g g\in\ZZ[R]$ throughout this section, where $d_g$ is the multiplicity of $g \in A$.
It is well known that the Teichm\"{u}ller system $T$ is a $(2^n,2^n,2^n,1)$-RDS in $R$ with respect to $Z$, where
$$Z=\{2x\mid x \in R\}.$$
Bonnecaze and Duursma in \cite{BD} showed that $T$ gives rise to an association scheme. More specifically, when $n \ge 3$, we have four disjoint subsets
$$
\Om_0=\{0\},\
\Om_1=T^*,\
\Om_2=\{-x \mid x \in \Om_1\},\
\Om_3=Z \setminus \{0\},
$$
where $T^*:=T \setminus \{0\}$.
The rest elements of $R$ are divided into two classes. Let $\Om_4$ contain the remaining ones which appear in the multiset $T^2$ and let $\Om_5$ contain the remaining ones which do not. The partition $\{\Om_i \mid 0 \le i \le 5\}$ forms a Schur ring over $R$, which leads to a $5$-class association scheme.
For a pseudo-planar function $f$, the set
$$D_f=\{x+2\sqrt{f(x)} \mid x \in T\}$$
is also a $(2^n,2^n,2^n,1)$-RDS in $R$ with respect to $Z$ (see \cite{SZh2013}).
Consequently, it is natural to ask whether an association scheme can also be obtained from $D_f$ or not. In this section, we prove that any relative difference set $D_f$, which necessarily arises from a pseudo-planar function $f$, will produce an association scheme. In fact, the partition of $R$ is obtained in a similar way. At first, we have four subsets
$$
\Sc_0=\{0\},\
\Sc_1=D_f \setminus \{0\},\
\Sc_2=\{-x \mid x \in \Sc_1\}=\Sc_1^{(-1)},\
\Sc_3=Z \setminus \{0\}.
$$
Furthermore, the remaining elements of $R$ are divided into two classes. Let $\Sc_4$ contain the remaining ones which appear in the multiset $D_f^2$ and let $\Sc_5$ contain the remaining ones which do not.

Using the following lemma, it is straightforward to verify that $\{\Sc_i \mid 0 \le i \le 5\}$ indeed forms a partition of $R$.

\begin{lemma}[{\cite[Theorem 1]{BD}}]
\label{lemma1}
Let $R=\GR$ and $T$ be the Teichm\"{u}ller system.
\begin{enumerate}
\item The multiset $TT^{(-1)}$ contains $0$ with multiplicity $2^n$, no other elements of $Z$, and the elements outside $Z$ with multiplicity one.
\item The multiset $T^2$ contains the elements of $Z$ with multiplicity one, and half of the elements outside $Z$ with multiplicity two.
\end{enumerate}
\end{lemma}

Now we consider the dual partition of $\{\Sc_i \mid 0 \le i \le 5\}$ on the character group $\widehat{R}$.
According to \cite[Theorem 3]{SZh2013}, if $f$ is pseudo-planar then
$\X(D_f)$ takes six values when $\chi$ ranges over $\widehat{R}$.
More precisely,
$$
\chi_a(D_f)=\begin{cases}
                          2^n & \text{for $a=0$ },\\
                          0   & \text{for $a \in Z \setminus \{0\}$},\\
                          \pm2^{(n-1)/2}\pm2^{(n-1)/2}\mathbf{i} & \text{for $a \in R \setminus Z$},
            \end{cases}
$$
when $n$ is odd and
$$
\chi_a(D_f)=\begin{cases}
                          2^n & \text{for $a=0$ },\\
                          0   & \text{for $a \in Z \setminus \{0\}$},\\
                          \text{$\pm2^{n/2}$ or $\pm2^{n/2}\mathbf{i}$} & \text{for $a \in R \setminus Z$},
             \end{cases}
$$
when $n$ is even. Furthermore, it is natural to investigate the
frequencies of these six values when $\X$ ranges over $\widehat{R}$.
Similar to the case of almost perfect nonlinear functions, we
introduce the definition of Fourier spectrum of a pseudo-planar function
$f$ as follows.

\begin{definition}
The Fourier spectrum of a pseudo-planar function $f$ is defined to be the
multiset
$$
\{\X(D_f) \mid \X \in \widehat{R}\}.
$$
\end{definition}
As a consequence of Theorem \ref{as} below, we can show that the
Fourier spectrum is the same for every pseudo-planar function.

Note that $\X(\Sc_1)=\X(D_f)-1$. There is a natural partition
$\{\E_i \mid 0 \le i \le 5\}$ on the character group $\widehat{R}$,
where $\chi_a$ and $\chi_b$ are in the same class if and only if
$\chi_a(\Sc_1)=\chi_b(\Sc_1)$. The partition $\{\E_i  \mid 0 \le i
\le 5\}$ is given as follows:
\begin{equation}\label{oddPartition}
    \begin{array}{lll}
\E_0&=&\{\X_0\},\\
\E_1&=&\{\X \in \widehat{R} \mid \X(\Sc_1)=-1\}=\{\X_a  \mid a \in Z \setminus \{0\}\},\\
\E_2&=&\{\X \in \widehat{R} \mid \X(\Sc_1)=-1+2^{(n-1)/2}+2^{(n-1)/2}\mathbf{i}\},\\
\E_3&=&\{\X \in \widehat{R} \mid \X(\Sc_1)=-1+2^{(n-1)/2}-2^{(n-1)/2}\mathbf{i}\},\\
\E_4&=&\{\X \in \widehat{R} \mid \X(\Sc_1)=-1-2^{(n-1)/2}+2^{(n-1)/2}\mathbf{i}\},\\
\E_5&=&\{\X \in \widehat{R} \mid
\X(\Sc_1)=-1-2^{(n-1)/2}-2^{(n-1)/2}\mathbf{i}\},
\end{array}
\end{equation}
when $n$ is odd and
\begin{equation}\label{evenPartition}
    \begin{array}{lll}
\E_0&=&\{\X_0\},\\
\E_1&=&\{\X \in \widehat{R} \mid \X(\Sc_1)=-1\}=\{\X_a  \mid a \in Z \setminus \{0\}\},\\
\E_2&=&\{\X \in \widehat{R} \mid \X(\Sc_1)=-1+2^{n/2}\},\\
\E_3&=&\{\X \in \widehat{R} \mid \X(\Sc_1)=-1-2^{n/2}\},\\
\E_4&=&\{\X \in \widehat{R} \mid \X(\Sc_1)=-1+2^{n/2}\mathbf{i}\},\\
\E_5&=&\{\X \in \widehat{R} \mid \X(\Sc_1)=-1-2^{n/2}\mathbf{i}\},
\end{array}
\end{equation}
when $n$ is even.

In the following we show that $(R,\{\Sc_i\}_{i=0}^5)$ is a Schur
ring, whose dual is $(\widehat{R},\{\E_i\}_{i=0}^5)$. We first use
Lemma~\ref{lemma2} and Lemma~\ref{lemma3} to prove that $\Sc_4$ can
be expressed as a linear combination of $\Sc_1^2,\Sc_2,$ and
$\Sc_3$. Then the values of $\chi(\Sc_4)$ and $\chi(\Sc_5)$ can be
determined where $\chi$ ranges over $\widehat{R}$. Combining this
with Bannai-Muzychuk criterion, the result follows.


\begin{lemma}
\label{lemma2}
Let $R=\GR$, and $f$ be a pseudo-planar function over $\F_{2^n}$ which can be identified with a map from
$T$ to $T$. Let $D_f=\{x+2\sqrt{f(x)} \mid x \in T\}$ and $\Sc_1=D_f \setminus \{0\}$.
\begin{enumerate}
\item The multiset $\Sc_1\Sc_1^{(-1)}$ consists of $0$ with multiplicity $2^n-1$ and the elements of $\Sc_4 \cup \Sc_5$ with multiplicity one.
\item The multiset $\Sc_1^2$ contains the elements of $\Sc_3$ with multiplicity one. In $\Sc_1^2$, the multiplicity of an element outside $\Sc_3$ is either zero or two.
\end{enumerate}
\end{lemma}
\begin{proof}
\begin{enumerate}
\item Since $f$ is pseudo-planar, the set $D_f$ is an RDS with $D_fD_f^{(-1)}=2^n \Sc_0+(R-Z)$. It is easy to verify that $\Sc_1\Sc_1^{(-1)}=(2^n-1)\Sc_0+(R-Z-\Sc_1-\Sc_2)=(2^n-1)\Sc_0+\Sc_4+\Sc_5$.
\item For any $x,y,z \in T^*$, suppose $\x+\y=2z$. Then $\x=y+2(\sqrt{f(y)}\oplus z \oplus y)$. By the unique representation (\ref{eqn}), we must have $x=y=z$. Hence $\Sc_1^2$ contains the elements of $\Sc_3$ with multiplicity one. Suppose $\Sc_1^2=\Sc_3+2U_f$, where $U_f=\sum_{g \in R \setminus \Sc_3} d_g g$, it suffices to show that $d_g=0$ or $1$. Since $\Sc_1^2=\Sc_3+2U_f$, applying the principal character, we have
    \begin{equation}
    \label{eqn1}
    \sum_{g \in R \setminus \Sc_3} d_g=(2^n-1)(2^{n-1}-1).
    \end{equation}
    Now, we consider the coefficient of $0$ in $\Sc_1^2(\Sc_1^{(-1)})^2$. On one hand, $\Sc_1^2(\Sc_1^{(-1)})^2=(\Sc_1\Sc_1^{(-1)})^2=((2^n-1)\Sc_0+\Sc_4+\Sc_5)^2=(2^n-1)^2\Sc_0+2(2^n-1)(\Sc_4+\Sc_5)+(\Sc_4+\Sc_5)^2$. 
    Since $\Sc_4+\Sc_5=\Sc_4^{(-1)}+\Sc_5^{(-1)}$ and $|\Sc_4 \cup \Sc_5|=(2^n-1)(2^n-2)$, we have $[(\Sc_4+\Sc_5)^2]_0=(2^n-1)(2^n-2)$. Consequently,
    $[\Sc_1^2(\Sc_1^{(-1)})^2]_0=(2^n-1)(2^{n+1}-3)$. On the other hand, $\Sc_1^2(\Sc_1^{(-1)})^2=(\Sc_3+2U_f)(\Sc_3+2U_f^{(-1)})=\Sc_3^2+2\Sc_3U_f+2\Sc_3U_f^{(-1)}+4U_fU_f^{(-1)}$. It is easy to check that $[\Sc_1^2(\Sc_1^{(-1)})^2]_0=2^n-1+4\sum_{g \in R \setminus \Sc_3}d_g^2$. Therefore, we have
    \begin{equation}
    \label{eqn2}
    \sum_{g \in R \setminus \Sc_3} d_g^2=(2^n-1)(2^{n-1}-1).
    \end{equation}
    By Equations~(\ref{eqn1})-(\ref{eqn2}), we have
    $$
    \sum_{g \in R \setminus \Sc_3} d_g=\sum_{g \in R \setminus \Sc_3} d_g^2,
    $$
    which implies that $d_g=0$ or $1$.
\end{enumerate}
\end{proof}

Now we proceed to determine $U_f$ mentioned in the proof of Lemma~\ref{lemma2}.

\begin{lemma}
\label{lemma3}
Let $R=\GR$ and $f$ be a pseudo-planar function over $\F_{2^n}$.
Let $\Sc_i,0\le i\le 5$ be defined as above.
Then we have
\begin{enumerate}
\item $\Sc_1^2=\Sc_3+2\Sc_4$ when $n$ is odd;
\item $\Sc_1^2=\Sc_3+2\Sc_2+2\Sc_4$ when $n$ is even.
\end{enumerate}
\end{lemma}
\begin{proof}
    We only present the proof for Assertion 2, because a similar method can be applied to Assertion 1. The partition $\{\E_i  \mid 0 \le i \le 5\}$ is given in \eqref{evenPartition}.
Define $m_i=|\E_i|$ for $0 \le i \le 5$, then $m_0=1$ and $m_1=2^n-1$. As a preparation, we first consider the relations between $m_2$, $m_3$, $m_4$ and $m_5$. A straightforward computation shows that $\sum_{a \in R} \X_a(D_f)=2^{2n}$. On the other hand,
\begin{eqnarray*}
\sum_{a \in R} \X_a(D_f) &=& m_0\cdot2^n+m_1\cdot0+m_2\cdot2^{n/2}+m_3\cdot(-2^{n/2})+m_4\cdot2^{n/2}\mathbf{i}+m_5\cdot(-2^{n/2}\mathbf{i})\\
                       &=& 2^n+2^{n/2}(m_2-m_3)+2^{n/2}(m_4-m_5)\mathbf{i}.
\end{eqnarray*}
Consequently, we have
\begin{align*}
m_2-m_3&=2^{3n/2}-2^{n/2},\\
m_4-m_5&=0.
\end{align*}
By Lemma \ref{lemma2}, $\Sc_1^2=\Sc_3+2U_f$. For any $x,y \in T$, if $\x+\y=0$, then $x=y+2(\sqrt{f(x)}\oplus\sqrt{f(y)}\oplus y)$. The latter equation implies $x=y=0$. Hence, $0$ is not an element of $\Sc_1^2$, i.e., $U_f \cap \Sc_0 = \emptyset$. By definition, we see that $\Sc_4 \subset U_f$ and $\Sc_5 \cap U_f = \emptyset$. It remains to determine the relationship between $\Sc_1$, $\Sc_2$ and $U_f$.

Firstly, we consider $\Sc_1$. By the inversion formula,
\begin{eqnarray*}
[D_f^2D_f^{(-1)}]_0 &=& \frac{1}{|R|}\sum_{a \in R} \X_a(D_f^2D_f^{(-1)})\\
          &=& \frac{1}{|R|}\sum_{a \in R} |\X_a(D_f)|^2 \X_a(D_f)\\
          &=& \frac{1}{|R|}(2^{3n}+2^{3n/2}(m_2-m_3)+2^{3n/2}(m_4-m_5)\mathbf{i})\\
          &=& 2^{n+1}-1.
\end{eqnarray*}
Note that $$D_f^2D_f^{(-1)}=\Sc_1^2\Sc_1^{(-1)}+2\Sc_1\Sc_1^{(-1)}+\Sc_1^2+2\Sc_1+\Sc_2+\Sc_0,$$ $[\Sc_1\Sc_1^{(-1)}]_0=2^n-1$ and $[\Sc_0]_0=1$. It follows that $[\Sc_1^2\Sc_1^{(-1)}]_0$=0. Hence, $\Sc_1^2$ contains no element of $\Sc_1$, namely, $\Sc_1 \cap U_f = \emptyset$.

Secondly, we consider $\Sc_2$. By the inversion formula,
\begin{eqnarray*}
[D_f^3]_0 &=& \frac{1}{|R|}\sum_{a \in R} \X_a(D_f)^3\\
          &=& \frac{1}{|R|}(2^{3n}+2^{3n/2}(m_2-m_3)-2^{3n/2}(m_4-m_5)\mathbf{i})\\
          &=& 2^{n+1}-1.
\end{eqnarray*}
From $$D_f^3=(\Sc_0+\Sc_1)^3=\Sc_0+3\Sc_1+3\Sc_1^2+\Sc_1^3,$$ $[\Sc_0]_0=1$, and $[\Sc_1]_0=[\Sc_1^2]_0=0$, it follows that $[\Sc_1^2\Sc_2^{(-1)}]_0=[\Sc_1^3]_0=2^{n+1}-2$. By Lemma \ref{lemma2}, $\Sc_1^2$ contains each element of $\Sc_2$ with multiplicity at most two. On the other hand, we have $[\Sc_1^2\Sc_2^{(-1)}]_0=2|\Sc_2|$.
Hence, each element of $\Sc_2$ appears in $\Sc_1^2$ with multiplicity exactly two. Therefore, when $n$ is even, we have $\Sc_1^2=\Sc_3+2\Sc_2+2\Sc_4$.
\end{proof}

The partition $\{\Sc_i \mid 0 \le i \le 5\}$ of $R$ induces a partition $\{\R_i \mid 0 \le i \le 5\}$ of $R \times R$, where
$$
\R_i=\{(x,y) \in R \times R \mid x-y \in \Sc_i\} \quad (0 \le i \le 5).
$$
Now we are ready to prove that $(R, \{\R_i\}_{i=0}^5)$
indeed forms an association scheme.

\begin{theorem}
\label{as}
Let $R=\GR$ and $\Sc_i,0\le i\le 5$ be defined as above. Then $(R,\{\Sc_i\}_{i=0}^5)$ is a Schur ring, whose dual is $(\widehat{R},\{\E_i\}_{i=0}^5)$. If $n \ge 3$, then $(R, \{\R_i\}_{i=0}^5)$ forms a $5$-class association scheme, whose first eigenmatrix is given as follows.
When $n$ is odd, suppose $b=2^{(n-1)/2}$, we have
\begin{equation}
\label{eqn3}
P=\left[ \begin {array}{cccccc} 1&2\,{b}^{2}-1&2\,{b}^{2}-1&2\,{b}^{2}-
1&2\,{b}^{4}-3\,{b}^{2}+1&2\,{b}^{4}-3\,{b}^{2}+1\\ \noalign{\medskip}
1&-1&-1&2\,{b}^{2}-1&-{b}^{2}+1&-{b}^{2}+1\\ \noalign{\medskip}1&-1+b+
b\mathbf{i}&-1+b-b\mathbf{i}&-1& \left( 1-b \right)  \left( 1-b\mathbf{i} \right) & \left( 1-b
 \right)  \left( 1+b\mathbf{i} \right) \\ \noalign{\medskip}1&-1+b-b\mathbf{i}&-1+b+b\mathbf{i}&-
1& \left( 1-b \right)  \left( 1+b\mathbf{i} \right) & \left( 1-b \right)
 \left( 1-b\mathbf{i} \right) \\ \noalign{\medskip}1&-1-b+b\mathbf{i}&-1-b-b\mathbf{i}&-1&
 \left( 1+b \right)  \left( 1-b\mathbf{i} \right) & \left( 1+b \right)  \left(
1+b\mathbf{i} \right) \\ \noalign{\medskip}1&-1-b-b\mathbf{i}&-1-b+b\mathbf{i}&-1& \left( 1+b
 \right)  \left( 1+b\mathbf{i} \right) & \left( 1+b \right)  \left( 1-b\mathbf{i}
 \right) \end {array} \right].
\end{equation}
When $n$ is even, suppose $b=2^{(n-2)/2}$, we have
\begin{equation}
\label{eqn4}
P=\left[ \begin {array}{cccccc} 1&4\,{b}^{2}-1&4\,{b}^{2}-1&4\,{b}^{2}-
1&8\,{b}^{4}-10\,{b}^{2}+2&8\,{b}^{4}-2\,{b}^{2}\\ \noalign{\medskip}1
&-1&-1&4\,{b}^{2}-1&-2\,{b}^{2}+2&-2\,{b}^{2}\\ \noalign{\medskip}1&2
\,b-1&2\,b-1&-1&2\,{b}^{2}-4\,b+2&-2\,{b}^{2}\\ \noalign{\medskip}1&-2
\,b-1&-2\,b-1&-1&2\,{b}^{2}+4\,b+2&-2\,{b}^{2}\\ \noalign{\medskip}1&-
1+2\,b\mathbf{i}&-1-2\,b\mathbf{i}&-1&-2\,{b}^{2}+2&2\,{b}^{2}\\ \noalign{\medskip}1&-1-
2\,b\mathbf{i}&-1+2\,b\mathbf{i}&-1&-2\,{b}^{2}+2&2\,{b}^{2}\end {array} \right].
\end{equation}
The second eigenmatrix is listed in Appendix.
\end{theorem}
\begin{proof}
According to the Bannai-Muzychuk criterion, it suffices to prove that $\X_j(\Sc_i)$ is a constant for any $\X_j \in \E_j$, where $0 \le i,j \le 5$. This is trivially true for any $0 \le j \le 5$ and $0 \le i \le 3$, which can be verified by direct computations. By Lemma \ref{lemma3}, we can obtain $\X_j(\Sc_4)$ for any $0 \le j \le 5$. Then we get the values of $\X_j(\Sc_5)$. The information of $\X_j(\Sc_4)$ and $\X_j(\Sc_5)$ completes the proof.
\end{proof}

\begin{remark}
\item[(1)] When $n=1$, we have $\Sc_4=\Sc_5=\emptyset$. Then $(R, \{\R_i\}_{i=0}^5)$ is a $3$-class association scheme. When $n=2$, we get $\Sc_4=\emptyset$. Then $(R, \{\R_i\}_{i=0}^5)$ forms a $4$-class association scheme, whose first eigenmatrix can be easily determined as a submatrix of (\ref{eqn3}) or (\ref{eqn4}).
\item[(2)] The $5$-class association scheme investigated in \cite{BD} can be regarded as a special case of our construction where the pseudo-planar function $f=0$.
\end{remark}

\begin{corollary}\label{FourierSpectrum}
Suppose $f$ is a pseudo-planar function over $\F_{2^n}$. Then the Fourier spectrum $\{\X(D_f) \mid \chi \in \widehat{R}\}$ is that listed in Tables \ref{table1} or  \ref{table2}.
\end{corollary}
\begin{proof}
Note that the frequency of each value can be obtained from the cardinality of the set $|\E_i|$. According to the second eigenmatrices listed in Appendix, the result now follows.
\end{proof}

\begin{table}
\begin{center}
\caption{Fourier spectrum, $n$ odd, $b=2^{(n-1)/2}$}
\label{table1}
\begin{tabular}{|c|c|}

\hline
Value  &  Frequency \\ \hline
$2b^2$  &  $1$ \\ \hline
$0$     &  $2b^2-1$ \\ \hline
$b+b\mathbf{i}$ & $\frac{b(2b^3+2b^2-b-1)}{2}$ \\ \hline
$b-b\mathbf{i}$ & $\frac{b(2b^3+2b^2-b-1)}{2}$ \\ \hline
$-b+b\mathbf{i}$ & $\frac{b(2b^3-2b^2-b+1)}{2}$ \\ \hline
$-b-b\mathbf{i}$ & $\frac{b(2b^3-2b^2-b+1)}{2}$ \\ \hline
\end{tabular}
\end{center}
\end{table}

\begin{table}
\begin{center}
\caption{Fourier spectrum, $n$ even, $b=2^{(n-2)/2}$}
\label{table2}
\begin{tabular}{|c|c|}
\hline
Value  &  Frequency \\ \hline
$4b^2$  &  $1$ \\ \hline
$0$     &  $4b^2-1$ \\ \hline
$2b$ & $b(4b^3+4b^2-b-1)$ \\ \hline
$-2b$ & $b(4b^3-4b^2-b+1)$ \\ \hline
$2b\mathbf{i}$ & $b^2(4b^2-1)$ \\ \hline
$-2b\mathbf{i}$ & $b^2(4b^2-1)$ \\ \hline
\end{tabular}
\end{center}
\end{table}

\section{Concluding remarks}\label{sec:concluding}

In this paper, three new classes of pseudo-planar binomial functions are provided.
In addition, we present a class of association schemes derived from
pseudo-planar functions, which can be considered as a natural generalization of
the one studied in \cite{BD}.

Let $D_1, D_2\subset G$ be two $(2^n,2^n,2^n,1)$ relative difference sets. They are
{\em equivalent} if there exist some $\alpha\in\Aut(G)$ and $a\in G$ such
that $\alpha(D_1)=D_2+a$.
Suppose $f$ is a function from $\F_{2^n}$ to itself.
It is proved in \cite{SZh2013} that $D_f$ is a $(2^n,2^n,2^n,1)$-RDS in $R=\GR$ with respect to $Z$ if and only if $f$ is pseudo-planar. So we say that two pseudo-planar functions $f_1$ and $f_2$ are {\em equivalent} if the relative difference
sets $D_{f_1}$ and $D_{f_2}$ are equivalent. By Corollary~\ref{FourierSpectrum}, the $p$-ranks and Smith
normal forms of the relative difference set $D_f$ associated with pseudo-planar functions are all the same.
Therefore some other techniques are to be developed to solve the equivalence problem.
The equivalence problem of pseudo-planar functions will be investigated in a manuscript prepared by Yue Zhou.

The following are several open problems.

\begin{enumerate}
\item All pseudo-planar binomials constructed in this paper are of type
      $$f(x)=ax^{2^i+2^j}+bx^{2^k+2^l}, $$
      where $i\ne j, k\ne l$, and $\{i,j\}\ne \{k,l\}$.
      For $n \le 9$, an exhaustive computer search shows that these pseudo-planar
      binomials can only exist on the finite field of the form
      $\F_{2^n}=\F_{2^{3m}}$. Therefore, it is interesting to examine that
      whether these pseudo-planar binomials can only exist in $\F_{2^n}$ with
      $3|n$ or not.

    \item The necessary and sufficient condition we provided in
      Proposition~\ref{prop:construction2}
      is not easily handled. It is desirable if one can derive a simpler
      characterization.
\end{enumerate}

\section*{Acknowledgements}
The authors express their gratitude to the anonymous
reviewers for their detailed and constructive comments which
are very helpful to the improvement of this paper, and to Prof. Claude Carlet, the Associate Editor, for his excellent editorial
job. We would like to thank Professor Qing Xiang and the anonymous reviewer for suggestions on the proofs of
Proposition~\ref{prop:construction3} and Proposition~\ref{prop:construction4},
and Dr. Yue Zhou for valuable comments and suggestions.
S. Hu was supported by the Scholarship Award for Excellent Doctoral Student
granted by Ministry of Education. T. Feng was supported in part by Fundamental Research Fund for the Central Universities of China, Zhejiang Provincial Natural Science Foundation under Grant No.~
LQ12A01019, in part by the National Natural Science Foundation of China
under Grant No.~11201418, and in part by the Research Fund for Doctoral Programs
from the Ministry of Education of China under Grant No.~20120101120089.
G. Ge was supported by the National Natural Science Foundation of China under
Grant No.~61171198 and Zhejiang Provincial Natural
Science Foundation of China under Grant No.~LZ13A010001.
\section*{Appendix}\label{sec:appendix}

When $n$ is odd, the second eigenmatrix of the association scheme is
\begin{small}
\begin{equation*}
Q=\left[ \begin{array}{cccccc} 1&2\,{b}^{2}-1&\frac{b}{2} \left( 2\,{b}^{3}+2\,{b}^{2}-b-1 \right) &\frac{b}{2} \left( 2\,{b}^{3}+2\,{b}^{2}-b-1 \right) & \frac{b}{2} \left( 2\,{b}^{3}-2\,{b}^{2}-b+1 \right) &\frac{b}{2} \left( 2\,{b}^{3}-2\,{b}^{2}-b+1 \right) \\
\noalign{\medskip}1&-1&\frac{b}{2} \left( {b}^{2}-1-({b}^{2}+b)\mathbf{i} \right) &\frac{b}{2} \left( {b}^{2}-1+({b}^{2}+b)\mathbf{i} \right) & \frac{b}{2} \left( 1-{b}^{2}-({b}^{2}-b)\mathbf{i} \right) &\frac{b}{2} \left( 1-{b}^{2}+({b}^{2}-b)\mathbf{i} \right) \\
\noalign{\medskip}1&-1&\frac{b}{2} \left( {b}^{2}-1+({b}^{2}+b)\mathbf{i} \right) &\frac{b}{2} \left( {b}^{2}-1-({b}^{2}+b)\mathbf{i} \right) & \frac{b}{2} \left( 1-{b}^{2}+({b}^{2}-b)\mathbf{i} \right) &\frac{b}{2} \left( 1-{b}^{2}-({b}^{2}-b)\mathbf{i} \right) \\
\noalign{\medskip}1&2\,{b}^{2}-1&-\frac{b}{2} \left( 1+b \right) &-\frac{b}{2} \left( 1+b \right) & \frac{b}{2} \left( 1-b \right) &\frac{b}{2} \left( 1-b \right) \\
\noalign{\medskip}1&-1&-\frac{b}{2} \left( 1+b\mathbf{i} \right) &\frac{b}{2} \left( -1+b\mathbf{i} \right) & \frac{b}{2} \left( 1+b\mathbf{i} \right) &\frac{b}{2} \left( 1-b\mathbf{i} \right) \\
\noalign{\medskip}1&-1&\frac{b}{2} \left( -1+b\mathbf{i} \right) &-\frac{b}{2} \left( 1+b\mathbf{i} \right) & \frac{b}{2} \left( 1-b\mathbf{i} \right) &{\frac {b \left( {b}^{2}+1 \right) }{2(1-b\mathbf{i})}} \end{array} \right].
\end{equation*}
\end{small}
When $n$ is even, the second eigenmatrix of the association scheme is
\begin{equation*}
Q=\left[ \begin {array}{cccccc} 1&4\,{b}^{2}-1&b \left( 4\,{b}^{3}-b+4\,{b}^{2}-1 \right) &b \left( 4\,{b}^{3}-4\,{b}^{2}-b+1 \right) & {b}^{2} \left( 4\,{b}^{2}-1 \right) &{b}^{2} \left( 4\,{b}^{2}-1 \right) \\
\noalign{\medskip}1&-1&b \left( b+2\,{b}^{2}-1 \right) &- \left( 2\,{b}^{2}-b-1 \right) b & -{b}^{2} \left( 1+2\,b\mathbf{i} \right) &{b}^{2} \left( -1+2\,b\mathbf{i} \right) \\
\noalign{\medskip}1&-1&b \left( b+2\,{b}^{2}-1 \right) &- \left( 2\,{b}^{2}-b-1 \right) b & {b}^{2} \left( -1+2\,b\mathbf{i} \right) &-{b}^{2} \left( 1+2\,b\mathbf{i} \right) \\
\noalign{\medskip}1&4\,{b}^{2}-1&-b \left( 1+b \right) &-b \left( -1+b \right) & -{b}^{2}&-{b}^{2}\\
\noalign{\medskip}1&-1&b \left( -1+b \right) &b \left( 1+b\right) & -{b}^{2}&-{b}^{2}\\
\noalign{\medskip}1&-1&-b \left( 1+b\right) &-b \left( -1+b \right) & {b}^{2}&{b}^{2}\end {array} \right].
\end{equation*}

\bibliographystyle{plain}
\bibliography{Ref}

\end{document}